\documentclass[12pt,reqno]{amsart}
\usepackage[letterpaper,margin=1in,footskip=0.25in]{geometry}
\usepackage{amssymb}
\usepackage[only,llbracket,rrbracket]{stmaryrd}
\usepackage{microtype}
\usepackage{mathtools}
\usepackage{enumitem}
\usepackage[noadjust]{cite}

\PassOptionsToPackage{dvipsnames}{xcolor}
\usepackage{tikz-cd}

\PassOptionsToPackage{pdfusetitle,colorlinks,pagebackref,linktocpage,bookmarksdepth=3}{hyperref}
\usepackage{bookmark}
\hypersetup{citecolor=OliveGreen,linkcolor=Mahogany,urlcolor=Plum}
\usepackage[hyphenbreaks]{breakurl}

\newcommand{\frakm}{{\mathfrak m}}

\newtheorem{theorem}{Theorem}[section]
\newtheorem{corollary}[theorem]{Corollary}
\newtheorem{lemma}[theorem]{Lemma}
\newtheorem{proposition}[theorem]{Proposition}

\newtheorem{innercustomthm}{Theorem}
\newenvironment{customthm}[1]
  {\renewcommand\theinnercustomthm{#1}\innercustomthm}
  {\endinnercustomthm}

\theoremstyle{definition}
\newtheorem{definition}[theorem]{Definition}

\theoremstyle{remark}
\newtheorem{remark}[theorem]{Remark}

\theoremstyle{question}
\newtheorem{question}[theorem]{Question}

\newtheoremstyle{cited}{.5\baselineskip\@plus.2\baselineskip\@minus.2\baselineskip}{.5\baselineskip\@plus.2\baselineskip\@minus.2\baselineskip}{\itshape}{}{\bfseries}{\bfseries .}{5pt plus 1pt minus 1pt}{\thmname{#1}\thmnumber{ #2}\thmnote{ \normalfont#3}}
\theoremstyle{cited}
\newtheorem{citedthm}[theorem]{Theorem}

\newtheorem{citedque}[theorem]{Question}

\newtheoremstyle{citeddef}{.5\baselineskip\@plus.2\baselineskip\@minus.2\baselineskip}{.5\baselineskip\@plus.2\baselineskip\@minus.2\baselineskip}{}{}{\bfseries}{\bfseries .}{5pt plus 1pt minus 1pt}{\thmname{#1}\thmnumber{ #2}\thmnote{ \normalfont#3}}
\theoremstyle{citeddef}
\newtheorem{citeddef}[theorem]{Definition}

\makeatletter
\def\l@subsection{\@tocline{2}{0pt}{2pc}{6pc}{}}
\makeatother

\begin{document}
\title[Vanishing of Tors of absolute integral closures in equicharacteristic zero]{Vanishing of Tors of absolute integral closures in equicharacteristic zero}

\author{Shravan Patankar}
\address{Department of Mathematics, Statistics and Computer Science\\University
of Illinois at Chicago\\Chicago, IL 60607-7045\\USA}
\email{\href{mailto:spatan5@uic.edu}{spatan5@uic.edu}}
%\urladdr{\url{https://spatan5.people.uic.edu/}}

\makeatletter
  \hypersetup{
    pdfauthor={Shravan Patankar},
    pdfsubject=\@subjclass,pdfkeywords={Absolute Integral Closures}
  }
\makeatother

\begin{abstract}
We show that a ring $R$ is regular if $Tor_{i}^{R}(R^{+},k) = 0$ for some $i\geq 1$ assuming further that $R$ is a $\mathbb{N}$-graded ring of dimension $2$ finitely generated over an equi-characteristic zero field $k$. This answers a question of Bhatt, Iyengar, and Ma. We use \emph{almost mathematics} over $R^{+}$ to deduce properties of the \emph{noetherian} ring $R$ and rational surface singularities. Moreover we show that $R^{+}$ in equi-characteristic zero is $m$-adically ideal(wise) separated, a condition which appears in the proof of local criterion for flatness. In dimension $2$ it is Ohm-Rush and intersection flat. As an application we show that the hypothesis can be astonishingly vacuous for $i \ll dim(R)$. We show that a positive answer to an old question of Aberbach and Hochster also answers this question. We use our techniques to make some remarks on a question of André and Fiorot regarding `fpqc analgoues' of splinters. 
\end{abstract}

\maketitle

\tableofcontents

\section{Introduction}
Throughout this article, all rings are assumed to be commutative and contain an identity element. The \emph{absolute integral closure} of an integral domain $R$, denoted by $R^{+}$, is the integral closure of $R$ inside an algebraic closure of its fraction field. In spite of being large and non-noetherian it is of great importance in commutative algebra and algebraic geometry. The purpose of this document is to give the first answers to the following question of Bhatt, Iyengar, and Ma:

\begin{citedque}[{\citeleft\citen{BIM19}\citemid end of Section 4\citeright }]\label{BIMq} If $(R,m,k)$ is a noetherian local domain
of equi-characteristic zero (i.e. $\mathbb{Q} \subset R$) and $Tor_{i}^{R}(R^{+},k) = 0$ for some $i \geq 1$, then is $R$ regular?
\end{citedque}

We show the following:
\begin{customthm}{\ref{thm:main}}
Let $R$ be a $\mathbb{N}$-graded ring of dimension $2$ finitely generated over an equi-characterstic zero field $k$. If $Tor_{i}^{R}(R^{+},k) = 0$ for some $i \geq 1$ then $R$ is regular.   
\end{customthm}
Henceforth we will often refer to equi-characteristic zero (i.e. $\mathbb{Q} \subset R)$ simply as characteristic zero. The work of Bhatt, Iyengar, and Ma has been the subject of several seminars and reading groups around the world because of it's connections to the direct summand theorem and perfectoid rings and Question \ref{BIMq} is the only question explicitly stated in it. The motivation for Question \ref{BIMq} arises naturally out of the following theorem of Bhatt, Iyengar, and Ma and the role absolute integral closures have played in commutative algebra and algebraic geometry.
\begin{citedthm}[{\citeleft\citen{BIM19}\citemid Theorem.\
  4.13\citeright }, \citeleft\citen{Bha21}\citemid Remark.\
  5.6\citeright] \label{BIMpm}
  Let $R$ be an excellent local domain of positive or mixed characteristic. If $Tor_{i}^{R}(R^{+},k) = 0$ for some $i \geq 1$, then $R$ is regular.
\end{citedthm}

We note that such vanishing of Tors type statements and questions are not new and are anticipated as the `only if' version of the above statement is true and follows from the seminal results of Hochster-Huneke \cite{HH94} and Bhatt \cite{Bha21} on Cohen-Macaualyness on absolute integral closures. That is, when $R$ is an excellent regular domain of positive characteristic the Hochster-Huneke result precisely says that $R^{+}$ is \emph{flat} over $R$ and Bhatt shows that when $R$ is an excellent local domain of mixed characteristic a system of parameters starting with $p$ is a \emph{Koszul regular} sequence on $R^{+}$, both of which are clearly \emph{stronger} statements than $Tor_{i}^{R}(R^{+},k) = 0$ for some $i \geq 1$. In particular the positive characteristic part of the above theorem is originally due to Aberbach and Li \cite{AL08} (who use completely different methods) and partial results can be found in the work of Aberbach \cite{Abe04} and Schoutens \cite{Sch03}  dating back to an explicit question of Huneke \cite[Exercise 8.8]{Hun96}. The importance of Hochster-Huneke's and Bhatt's theorems is hard to overstate. Bhatt's mixed characteristic analogue of this result has led to startling progress in the minimal model program in mixed characteristic and to quote Bhatt, the Hochster-Huneke result forms the bedrock of tight closure theory and a large amount of positive characteristic commutative algebra in general. These ideas also have fairly strong applications to complex geometry and algebraic geometry. 

The proof of our main theorem is of independent interest and we give a brief sketch in this paragraph. The hypothesis imply that every module finite normal extension of $R$ is flat. It follows by taking the normalization that $R$ is normal. Next, assume for the moment that $R$ has rational singularities. Since $R$ has dimension $2$ classical theory of rational surface singularities \cite{Lip78} implies that $R$ is $\mathbb{Q}$-Gorenstein and hence its cyclic cover is a module finite normal Gorenstein extension and by descent one can also conclude that $R$ is Gorenstein. This implies that $R$ is a rational double point and it is a classical fact that it has a module finite extension which is regular, and hence we have that $R$ is regular by descent. 

The direct summand conjecture was a famous conjecture in commutative algebra, open for fourty years until it was settled by Y. André in 2016 \cite{And18}. André's proof uses Scholze's theory of perfectoid spaces \cite{Sch12} and is based on an observation of Bhatt \cite{Bha16} which solves certain cases of the direct summand conjecture using Falting's almost purity theorem. Bhatt's idea, simply put, was to use a standard trick involving the interplay between almost mathematics and flat maps to transfer the direct summand conjecture to a problem about perfectoid spaces and rings which in particular are very large and non-noetherian. The \textit{pièce de résistance} of our proof is to use the same trick for the non-noetherian ring $R^{+}$ in characteristic zero to show that the hypothesis of the question imply that $R$ has rational singularities. There is evidence showing that this use of $R^{+}$ is necessary. Boutot's theorem implies that an equicharacteristic zero normal (equivalently a splinter) non-pseudorational ring has \emph{no} module finite extensions which are pseudorational. Hence one can't obtain rationality using the finite extensions techniques used to deduce other properties of $R$.

The reader may now guess why we described the proof. It is perhaps surprising that in spite of being a question purely in commutative algebra our proof makes use of the main theorems about rational surface singularities. This can be compared with previous results in the literature exhibiting surprising algebraic characterizations of rational singularities \cite{Ma18}, \cite{Har98}, \cite{MS97}. Moreover, this suggests a Kunz-type theorem in characteristic zero, i.e. $R$ is regular if and only if $Tor^{R}_{i}(R^{+}, k) = 0$ for some $i$ large enough. A similar theorem has been proved by Ma-Schwede \cite{MS20} who show that $R$ is regular if every alteration has `finite projective dimension' and our result implies `alterations' could be relaxed to `finite covers'. In their study of the behaviour of finite coverings of (affine) schemes in Grothendieck topologies André and Fiorot \cite{AF21} ask for rings which are ``fpqc analgoues of splinters''. It is shown using Kunz's theorem that in positive characteristic the only ``fpqc analgoues of splinters" are regular rings. Using our techniques we are able to make progress towards this question after imposing some conditions on the corresponding coverings. These conditions are closely related to Hochster and Huneke's proof of existence of big Cohen-Macaulay algebras in characteristic zero and ``big equational tight closure''. In higher dimensions we expect certain strengthenings of equational lemma (in the sense of Huneke and Lyubeznik \cite{HL07}) type statements to help make progress towards Question \ref{BIMq} and this raises questions even in positive characteristic. 

In a more homological direction, we show that $R^{+}$ is $m$-adically ideal separated in characteristic zero, a condition which essentially appears in the proof of miracle flatness or the local criterion of flatness. If $R$ is regular of dimension $2$ this implies $R^{+}$ is \emph{Ohm-Rush} or \emph{intersection flat}, notions recently studied by Hochster-Jefferies \cite{HJ21} and Epstein-Shapiro \cite{ES21}. These raise the interesting questions of whether the same holds in positive or mixed characteristic and imply that the hypothesis of Question \ref{BIMq} can be vacuous for $i \ll dim(R)$.

This paper is organized as follows. In Section $2$ we state and prove preliminaries. In Section $3$ we show that $R^{+}$ is $m$-adically ideal separated and that the hypothesis of the question can be vacuous for $i \ll dim(R)$. In Section $4$ we prove the main theorem and some corollaries. In Section $5$ we state natural questions which arise when we try to apply our techniques in higher dimensions. In Section $6$ we use our techniques to make remarks on a question of André and Fiorot \cite{AF21} and in particular answer it under very specific assumptions.

\section{Preliminaries}

In this section we gather several lemmas which we will need in the future sections. They are standard and have been used in works surrounding the homological conjectures and commutative algebra. We also briefly discuss general background regarding rational singularities, $m$-adic ideal separatedness and `fpqc analogues of splinters' in the sense of André and Fiorot.
\begin{lemma} 
Let $R$ be a normal characteristic zero domain. Then the map $R \rightarrow R^{+}$ and consequently $R\rightarrow S$ (where $S$ is any module-finite domain extension of $R$) splits.
\end{lemma}
\begin{proof}
Let $\theta$ be an element of $R^{+}$ satisfying a minimal polynomial of degree $d$. It is readily checked that the normalized trace map $\frac{1}{d}Tr: R[ \theta ] \rightarrow R$ gives a splitting of $R \rightarrow R[\theta]$ and hence we have a splitting of $R \rightarrow R^{+}$ by doing this for every element of $R^{+}$. 
\end{proof}

The following lemma is used several times in the proof of our main theorem.
\begin{lemma}
Let $(R,m,k)$ be an excellent local characteristic zero domain. Let $S$ be a module finite normal extension of $R$, which exists as $R$ is excellent. If $R$ satisfies the assumption of our main theorem, that is if $Tor^{R}_{i}(R^{+}, k) = 0$ for some $i$ then $S$ has projective dimension $dim(R) - i$ over $R$. In particular if $dim(R)=2$ then $S$ is flat over $R$.
\end{lemma}
\begin{proof}
The above lemma implies that $S\rightarrow S^{+} = R^{+}$ splits as an $S$-module and hence an $R$-module map. Since $Tor$ is a functor we have $Tor^{R}_{i}(S, k)$ is a direct summand of $Tor^{R}_{i}(R^{+}, k)$ and is hence equal to $0$. Since $S$ is a finitely generated module over a noetherian local ring $R$, standard theory of noetherian local rings implies $S$ has finite flat dimension and hence finite projective dimension. Since $R$ has dimension $2$ and $S$ has depth $2$ the statement of the Auslander-Buschbaum theorem implies that the projective dimension of $S$ is $0$ and hence $S$ is flat over $R$.
\end{proof}

The reader can guess that the above statement allows us to reduce the statement of the main theorem about the large non-noetherian object $R^{+}$ to a finitistic statement about finite normal covers of $R$. That is, it is not hard to see that the above lemma implies that if we have a ring $R$ satisfying the assumptions of the main theorem (that is if $Tor^{R}_{i}(R^{+}, k) = 0$ and $dim(R) = 2$) then every normal extension of $R$ is flat over $R$.

Let $P$ be a property of rings. $P$ is said to descend under faithfully flat maps if for a faithfully flat map of noetherian local  rings $R\rightarrow S$, $S$ is $P$ implies that $R$ is $P$. Many singularities and properties of rings descend under faithfully flat maps and we refer the reader to the work of Datta and Murayama \cite{DM20} for an excellent exposition surrounding this topic. We will need that the properties of a ring being normal, Gorenstein, and regular descend along faithfully maps.

\begin{theorem}
Let $R \rightarrow S$ be a faithfully flat map of noetherian local rings. If $S$ is normal so is $R$.
\end{theorem}
\begin{proof}
\cite[Cor. to Thm. 23.9]{Mat89}.
\end{proof}

\begin{theorem}
Let $R \rightarrow S$ be a faithfully flat map of noetherian local rings. If $S$ is Gorenstein so is $R$.
\end{theorem}
\begin{proof}
\cite[Cor. to Thm. 23.4]{Mat89}.
\end{proof}

\begin{theorem}
Let $R \rightarrow S$ be a faithfully flat map of noetherian local rings. If $S$ is regular so is $R$.
\end{theorem}
\begin{proof}
\cite[Cor. to Thm. 23.7]{Mat89}.
\end{proof}

\begin{theorem}
Let $R \rightarrow S$ be a cyclically pure homomorphism of locally quasi-excellent $\mathbb{Q}$-algebras, in particular $R\rightarrow S$ can split or be faithfully flat. If $S$ is pseudo-rational (i.e. has rational singularities) so does $R$.
\end{theorem}

\begin{proof}
This is \cite[Theorem C]{Mur21} and the theorem under stronger assumptions is due to Boutot (1987) and Schoutens (2008) and the reader can see Murayama's work for references.
\end{proof}

It is striking that our proof of a statement about the absolute integral closure or finite covers involves rational singularities which are defined using resolutions of singularities, which is a proper birational map and in particular very far from being a finite map. 

\begin{definition}
Let $X$ be an excellent scheme which admits a dualizing complex and let $f: Y \rightarrow X$ be a resolution of singularities, that is $Y$ is non-singular and the map $f: Y \rightarrow X$ is proper birational. $X$ is said to have (resolution) rational singularities if $R_{\ast}f(\mathcal{O}_{Y})\simeq \mathcal{O}_{X}$ and $f_{\ast}\mathcal{O}_{Y} \simeq \mathcal{O}_{X}$ .
\end{definition}

Rational singularities are one of the most widely studied singularities in singularity theory. Indeed to quote Kovacs, ``rational singularities are arguably one of the most mild and useful classes of singularities one can imagine". We will need criteria for determining when certain graded rings or cones over smooth projective varieties have rational singularities. 

\begin{theorem}
Let $Y$ be a smooth projective variety of dimension $n$ and let $L$ be an ample line bundle. The affine cone over $Y$ with conormal $L$ is the affine algebraic variety     
\begin{align*}
    X = Spec\bigoplus_{m\geq 0}H^{0}(Y, L^{m})
\end{align*}
$X$ has rational singularities if and only if $H^{i}(Y, L^{m}) = 0$ for every $i>0$ and for every $m \geq0$.
\end{theorem}
\begin{proof}
This is classical, see \cite[Prop. 3.13]{Kol13}.
\end{proof}

\begin{theorem}\label{thm:Wat}
Let $R$ be a normal $\mathbb{N}$-graded ring which is finitely generated over a field $R_{0} = K$ of characteristic zero. The $a$-invariant of $R$ is the highest integer a such that $[H^{d}_{m}(R)]_{a}$ is nonzero. Then $R$ has rational singularities if and only if the open set $Spec(R) - {m_{R}}$ has only rational singularities and the ring $R$ is Cohen-Macaulay with $a(R) < 0$.
\end{theorem}
\begin{proof}
This is (by now) classical, see \cite[Theorem 2.2]{Wat83}
\end{proof}

Our main results naturally tie in with the notion of dagger closure. Dagger closure was defined as an attempt to provide a characteristic free ideal closure operation with properties similar to tight closure. Hochster-Huneke \cite[Page 236]{HH91} remark that a priori, each valuation may give a different dagger closure, and show that this is not the case if $R$ has positive characteristic by proving that dagger closure agrees with tight closure (\cite[Theorem 3.1]{HH91}) for complete local domains (see \cite[Proposition 1.8]{BS12} where it is shown that dagger closure and tight closure agree for domains essentially of finite type over an excellent local ring). We do not know if dagger closure depends on the choice of the valuation if $R$ has equicharacteristic zero or mixed characteristic. Although tight closure has been studied extensively in positive characteristic, dagger closure continues to be mysterious\footnote{Some progress towards understanding dagger closure has been made by St\"{a}bler in his thesis work \cite{Sta10}.} and we refer the reader to \cite[Questions 4.1, 4.2]{RSS07} for some basic questions about dagger closure that remain unanswered. The results of this note can be viewed as an application of the theory of dagger closure.

\begin{citeddef}[{\citeleft\citen{BS12}\citeright}]
Let $R$ be a domain and $I$ an ideal. Then an element $f$ belongs to the dagger closure $I^{\dagger}$ of $I$ if for every valuation $v$ of rank at most one on $R^{+}$ and every positive $\epsilon$ there exists $a \in R^{+}$ with $ v(a) < \epsilon$ and $af \in IR^{+}$.
\end{citeddef}

This definition of dagger closure is slightly more general than the original definition due to Hochster and Huneke \cite{HH91}. In \cite{HH91} dagger closure is defined only for complete local domains and under this definition both coincide since both coincide with tight closure in positive characteristic. 

\begin{citeddef}[{\citeleft\citen{BS12}\citeright}] 
Let $R$ be a $\mathbb{Q}$-graded domain. The map $ v \colon (R - {0}) \rightarrow Q$ sending
$f \in (R - {0})$ to deg $f_{i}$, where $f_{i}$ is the minimal homogeneous component of $f$ induces a valuation on $R$ with values in $Q$. This valuation will be referred to as the valuation induced by the grading.
\end{citeddef}

\begin{citeddef}[{\citeleft\citen{BS12}\citeright}] 
Let $R$ denote an $\mathbb{N}$-graded domain and let $I$ be an ideal of $R$. Let $v$ be the valuation on $R^{+GR}$ induced by the grading on $R$. Then an element $f$ belongs to the graded dagger closure $I^{\dagger GR}$ of an ideal $I$ if for all positive $\epsilon$ there
exists an element $a \in R^{+GR}$ with $v(a) < \epsilon$ such that $af \in IR^{+GR}$. If $R$ is not a
domain we say that $f \in I^{\dagger GR}$ if $f \in (IR/P)^{\dagger GR}$ for all minimal primes $P$ of $R$ (this is well defined since minimal primes are homogeneous).
\end{citeddef}

We review basics about the notions of $I$-adic separatedness.
\begin{definition}
Let $R$ be a ring, $I$ be an ideal, and $M$ be an $R$-module. Then $M$ is said to be \textbf{$I$-adically separated} if $\cap_{n\geq 1}I^{n}M = 0$.
\end{definition}
If $R$ is noetherian and $M$ is finitely generated and if $I$ is a proper ideal then $M$ is $I$-adically separated. Hence this notion is particulary interesting for non-noetherian modules.

This notion ($m$-adically separated) is significantly weaker than the notion of $m$-adically ideal separatedness. For example it is proved by Hochster \cite{Ho02} (see also \cite[Lemma 4.2]{Shi10}) that $R^{+}$ is $m$-adically separated but as is discussed in Section 3, $R^{+}$ being $m$-adically ideal separated (in positive characteristic) would give another proof of \ref{BIMpm} which apriori uses the deep Cohen-Macaulayness of $R^{+}$. 
\begin{definition}
Let $R$ be a noetherian ring, $J$ be an ideal, and $M$ be an $R$-module. Then $M$ is said to be \textbf{$J$-adically ideal(wise) separated} if $M\otimes I$ is $J$-adically separated for every ideal $I$ of $R$.
\end{definition}

The significance of this definition is in the following theorem known as `Local criterion for flatness'. The local criterion of flatness helps in characterizing etalé morphisms.

\begin{theorem}
Let $A$ be a noetherian ring, $I$ an ideal, and $M$ an $I$-adically ideal separated module. If $Tor_{1}^{A}(A/I, M) = 0$ then $M$ is flat as an $A$ module.
\end{theorem}

This is used in the proof of the following classical theorem due to Grothendieck, called `miracle flatness' by Brian Conrad.

\begin{theorem}
Let $R \rightarrow S$ be a local ring homomorphism between local Noetherian rings. If $S$ is flat over $R$, then
$dim(S) = dim(R) + dim(S/mS)$ where $m$ is the maximal ideal of $R$.
Conversely, if this dimension equality holds, if R is regular and if S is Cohen–Macaulay (e.g., regular), then S is flat over R.
\end{theorem}
Let us review some basics of `fpqc analogues of splinters' in the sense of Andre and Fiorot.

\begin{definition}
Let $X$ be a noetherian affine scheme. $X$ is said to be an `fpqc analogue of splinter' if every finite cover $S$ of $X$ is a covering for the fpqc topology. Equivalently if $X = Spec(R)$ then there exists an $S$-algebra $T$ for any module finite extension $S$ of $R$ such that $T$ is flat over $R$ and $R \rightarrow S.
\rightarrow T$.
\end{definition}

It is known that a $F$-finite noetherian ring of positive characteristic is an fpqc analogue of splinter only if it is regular. This follows from Kunz's theorem. We give a brief sketch of the proof. We have a faithfully flat algebra $T$ (by assumption) such that $R\rightarrow F_{*}(R) \rightarrow T$ and by base change we have that $F_{*}(R) \rightarrow T \rightarrow F_{*}(T)$ is flat and in particular $F_{*}(R) \rightarrow T$ is pure. This implies $R\rightarrow F_{*}(R)$ is flat.  The fact that excellent regular noetherian rings (of any characteristic) are fpqc analogues of splinters is deep and follows from the existence of `Big Cohen-Macaulay algebras' in the sense of Hochster \cite{Ho75}, Andre\cite{And18}.

\section{$m$-adic ideal separatedness and vacuous assumptions}

The goal of this section is to prove that the absolute integral closure, $R^{+}$, of an excellent equi-characteristic zero domain is $m$-adically ideal separated and to show that the assumptions of the question of Bhatt, Iyengar, Ma \ref{BIMq} can be \emph{vacuous} under certain mild assumptions. That is to say that there are \emph{no} complete local domains $(R,m,k)$ of equi-characteristic zero such that $Tor_{i}^{R}(R^{+}, k) = 0$ for $i\leq dim(R) - 3$. Here we show this only when $dim(R) = 4$ and leave higher dimensions as an excercise to the reader. It is easily seen that imposing $i \geq dim(R)$ in Question \ref{BIMq} fixes this issue.

\begin{theorem}\label{intflat}
Let $R$ is an excellent (or more generally Nagata) local domain of characteristic zero. Then $R^{+}$ is $m$-adically ideal separated as an $R$-module.  
\end{theorem}

\begin{proof}
It suffices to show (by definition) that $R^{+} \otimes I$ is $m$-adically separated for every ideal $I$ of $R$. Let $\alpha \in \cap_{n \geq 0} m^{n}I \otimes R^{+}$. We need to show $\alpha = 0$. Let $\alpha = \sum_{i=1}^{k}r_{i} \otimes m_{i}$ where $r_{i} \in R^{+}$ and $m_{i} \in I$. We may assume that all $r_{i}$ are contained in a finite normal extension of $R$, $S$, by simply adjoining $r_{i}$ to $R$ and then taking the normalization. We used here that $R$ is Nagata which implies that normalizations are module-finite. Since $S$ has characteristic zero (and $S^{+} = R^{+}$ the maps $S \rightarrow R^{+}$ and $I\otimes S \rightarrow I \otimes R^{+}$ split. $\alpha$ is in $m^{n}(R^{+} \otimes I)$ and hence is in $m^{n}(S \otimes I)$ for every $n$. This clearly implies $\alpha$ is $0$ as $S$ and $I$ are finitely generated $R$ modules and hence so is $S \otimes I$. All finitely generated modules over a noetherian local domain are separated by Krull's intersection theorem.   

\end{proof}

\begin{theorem}
Let $R$ be a equi-characteristic zero complete local domain such that $dim(R) \geq 4$. Then $Tor_{1}^{R}(R^{+}, k) \neq 0$. In particular the hypothesis of the question of Bhatt, Iyengar, and Ma can be vacuous if one does not ask for $i \geq dim(R)$ in Question \ref{BIMq}. 
\end{theorem}
\begin{proof}
First assume $R$ is regular. Suppose $Tor_{1}^{R}(R^{+}, k) = 0$. Local criterion of flatness \cite[Theorem 22.3]{Mat89} implies $R \rightarrow R^{+}$ is \emph{flat}. This implies that $R^{+}$ is Cohen-Macaulay which is well known to be false given that $R$ has characteristic zero. 

When $R$ is not regular we simply notice that the problem reduces to when it is. We get by the same argument above that $R^{+}$ is flat. Choose a prime $\mathfrak{p}$ of height at least $3$ such that $R_{\mathfrak{p}}$ is regular. Such a prime exists by the lemma below. Absolute integral closures localize and localization of a flat module is flat. Hence we have that $(R^{+})_{\mathfrak{p}} = (R_{\mathfrak{p}})^{+}$ is flat over $R_{\mathfrak{p}}$. This implies that $R^{+}$ is Cohen-Macaulay which is well known to be false given that $R$ has characteristic zero. See \cite[Introduction, last line of paragraph after Theorem 1.1]{Bha21} and one can follow \cite[Proposition 2.4]{ST21} for an explicit proof.
\end{proof}

We thank Kevin Tucker for conversations regarding the lemma below.

\begin{lemma}
Let $(R,m)$ be a noetherian normal local ring of dimension $2$ or more. Then there exists a prime $\mathfrak{p}$ of height $dim(R)-1$ such that $R_{\mathfrak{p}}$ is regular.
\end{lemma}
\begin{proof}
Let $I = \{\cap \mathfrak{q}| R_{\mathfrak{q}} \text{is not regular}\}$. If $I = m$ then any prime of height $dim(R) - 1$ works, and such a prime exists as complete local domains are catenary. If $I$ is not $m$ and the height of $I$ is $d$, then there are finitely many prime ideals of height $d$ containing $I$ which are the minimal primes by primary decomposition. Since $R$ is normal $d\geq 2$. Then $R$ has infinitely many primes of height $d$ (by \cite[Theorem 144]{Kap74}) so one can simply take $\mathfrak{p}$ to be one which is not minimal over and consequently does not contain $I$. 
\end{proof}

\begin{remark}
It is clear that the above argument is going to work with, for example, $Tor_{2}^{R}(R^{+}, k) =0$ and $dim(R) = 5$. While this does not imply $R^{+}$ is \emph{flat} over $R$, it does show that flat dimension of $R^{+}$ is 2 (since it is a direct limit of normal rings $S$ which have flat dimension $2$) which will force the depth of $R^{+}$ to be greater than or equal to $3$ (arguing similarly by Auslander-Buschbaum formula). One can localize to primes of height $3$ to see that the depth of $R^{+}$ cannot be greater than $2$. We omit this discussion for the sake of clarity and brevity.
\end{remark}

Hence we have the following `corrected' version of the question of Bhatt, Iyengar, and Ma:
\begin{question}\label{BIMc}
If $(R,m,k)$ is a noetherian local domain
of equi-characteristic zero (i.e. $\mathbb{Q} \subset R$) and $Tor_{i}^{R}(R^{+},k) = 0$ for some $i \geq dim(R)$, then is $R$ regular?
\end{question}
This raises the following \emph{fascinating} question in other characteristics.
\begin{question}
Let $R$ be a complete local domain of positive characteristic. Is $R^{+}$ $m$-adically ideal separated? If $R$ is of mixed characteristic is the $p$-adic completion of $R^{+}$, $\widehat{R^{+}}$ $m$-adically ideal separated? 
\end{question}

Following the arguments of the proof of the theorem above we see that a positive answer to the above question would have `another' proof of the theorem of Bhatt, Iyengar, and Ma atleast for $Tor_{1}$ -  $Tor_{1}^{R}(R^{+}, k) = 0$ implies $R^{+}$ is flat which by `Kunz's theorem' implies $R$ is regular \cite[Theorem 4.7]{BIM19}. However one can check that the proof of the theorem of Bhatt, Iynegar, and Ma (for $Tor_{1}$) uses Cohen-Macaulayness of $R^{+}$, hence we expect the answer to the above question to use deep properties of $R^{+}$. It is easy to see why we ask for the $p$-adic completion of $R^{+}$ in mixed characteristic instead of $R^{+}$. $R^{+}$ is not $m$-adically ideal separated in mixed characteristic, we know that $Tor_{1}^{R}(R^{+}, k) = 0$, hence if true the local criterion for flatness would imply $R^{+}$ is flat and hence Cohen-Macaulay. It is well known that $R^{+}$ is not Cohen-Macaulay in mixed charactertistic.

The perfection of a ring $R$ of positive characteristic, $R_{perf}$ is $lim(R\rightarrow R \rightarrow R \rightarrow \cdots)$ where each map is the frobenius. It is sometimes denoted as $R^{\frac{1}{p^{\infty}}}$ in literature and conceptually is the large ring containing all $p$-th power roots of elements of $R$. It is easy to see that our proof goes through for algebras which are a limit of split noetherian ring maps. For example our proof goes through to show that the perfection of a positive characteristic $F$-split ring is $m$-adically ideal separated. This raises the question:
\begin{question}
Let $R$ be a noetherian domain of positive characteristic. Is $R_{perf}$ $m$-adicaly ideal separated?
\end{question}

\begin{remark}
The notion of $m$-adic ideal separatedness is closely related to the notion of weakly intersection flatness \cite[Propsition 5.7 e)]{HJ21} for ideals and Ohm-Rush modules \cite{ES16}, \cite{ES19}, \cite{ES21}. While it does not make sense to ask for intersection flatness of $R^{+}$ in characteristic zero as the results of this section show that $R^{+}$ is (usually) never flat over $R$, it does make sense to ask if $R^{+}$ is weakly intersection flat for ideals when $R$ has characteristic zero and whether ($p$-adic completion of ) $R^{+}$ is Ohm-Rush or intersection flat if $R$ is regular of positive or mixed characteristics. These are under investigation.
\end{remark}

We avoid stating the definitions of weakly intersection flatness, intersection flatness, and Ohm-Rush modules as we feel they are technical for the purposes of this section. We do mention the following proposition due to Hochster and Jefferies which says, roughly speaking, that when $R$ is a complete local ring and $S$ is an $R$-flat algebra, $m$-adic ideal separatedness with ideals replaced by all finitely generated modules is equivalent to intersection flatness. The forward implication is non-trivial.

\begin{proposition}
(ref. \cite[Proposition 5.7 e)]{HJ21} Let $R$ be a complete local ring and $S$ is an $R$-flat algebra. If $S \otimes M$ is $m$-adically separated for all finitely generated modules $M$ then $S$ is intersection flat as an $R$ module.
\end{proposition}

We make the following observation as an application of our techniques:
\begin{theorem}
Let $R$ be a regular complete local characteristic zero domain of dimension $2$ (a power series ring in two variables by Cohen structure theorem). Then $R^{+}$ is \emph{intersection flat} and Ohm-Rush as a $R$-module. 
\end{theorem}
\begin{proof}
This follows from our observation that in characteristic two $R^{+}$ is a `limit of split (noetherian) maps' and that in dimension $2$ $R \rightarrow R^{+}$ is flat when $R$ is regular. More precisely, since $R^{+}$ is Cohen-Macaulay in dimension $2$ it is flat over $R$ if $R$ is regular, in any characteristic. 

It suffices to show (by \cite[Proposition 5.7 e)]{HJ21} that $R^{+} \otimes M$ is $m$-adically separated for every finitely generated module $M$. Let $\alpha \in \cap_{n \geq 0} R^{+} \otimes M$. We need to show $\alpha = 0$. Let $\alpha = \sum_{i=1}^{k}r_{i} \otimes m_{i}$ where $r_{i} \in R^{+}$ and $m_{i} \in M$. We may assume that all $r_{i}$ are contained in a finite normal extension of $R$, $S$, by simply adjoining $r_{i}$ to $R$ and then taking the normalization. We used here that $R$ is Nagata which implies that normalizations are module-finite. Since $R, S, R^{+}$ all have characteristic zero the maps $S \rightarrow R^{+}$ and $M\otimes S \rightarrow M \otimes R^{+}$ split. $\alpha$ is in $m^{n}(R^{+} \otimes M)$ and hence is in $m^{n}(S \otimes M)$ for every $n$. This clearly implies $\alpha$ is $0$ as $S$ and $I$ are finitely generated $R$ modules and hence so is $S \otimes I$. All finitely generated modules over a noetherian local ring are separated by Krull's intersection theorem. 
\end{proof}

Note that our main result is in a sense a converse to this statement. That is if $R \rightarrow R^{+}$ is flat (under the stated assumptions on $R$) then $R$ is regular. One can similarly prove that if $R$ is regular and $F$-finite of positive characteristic then $R_{perf}$ is intersection flat. This fact was observed independently and earlier by Neil Epstein.

The proposition above raises the question:
\begin{question}
Let $R$ be a regular complete local domain of positive characteristic. Is $R^{+}$ intersection flat or Ohm-Rush? If $R$ is regular of mixed characteristic is the $p$-adic completion of $R^{+}$, $\widehat{R^{+}}$ intersection flat or Ohm-Rush? 
\end{question}
We have to assume $R$ is regular in the above question as intersection flat/ Ohm-Rush modules are flat and $R^{+}$ will be flat over $R$ only if $R$ is regular by the main theorems of the work of Bhatt, Iyengar, and Ma \cite[Theorems 4.7, 4.12]{BIM19}. The proof of our theorems above show that if $R^{+}$ is a `limit of split maps' then the answer is yes, however this is known to be false in both positive and mixed characteristics. If true it would imply every complete local domain has a module finite extension which is a splinter, however splinters are Cohen-Macaulay (this follows from the theorems of Hochster-Huneke and Bhatt on Cohen-Macaulayness of $R^{+}$) and it is known that `small Cohen-Macaulay algebras' do not exist \cite{ST21}. Hence it seems that a proof of this statement will require almost mathematics and other non-trivial techniques. 

We now show that a positive answer to an old question of Aberbach and Hochster answers the question of Bhatt, Iyengar, and Ma.

\begin{citedque}[{\citeleft\citen{AH97}\citemid Question 3.7\citeright }]\label{AHq}
If $R$ is a complete local domain containing $\mathbb{Q}$ then is the $Tor$ dimension of $R^{+}/m_{R^{+}}$ (as an $R^{+}$ module) equal to $dim(R)$?
\end{citedque}

We make some remarks on this question. It is fairly straightforward and was observed in \cite[Remark 4.5]{Pat22} that the answer to this question is yes if $dim(R) = 1$. This follows from the fact that under this assumption $R^{+}$ is a valuation domain. Hence the maximal ideal $m_{R^{+}}$ is flat as an $R^{+}$ module. The projective dimension of $R^{+}/m_{R^{+}}$ however is $2$. We briefly state the remarks Aberbach and Hochster make regarding Question \ref{AHq}.

\begin{lemma}
  Let $(R,m)$ be a quasi local ring of dimension $d$ such that every $d$-element $m$-primary ideal is a regular sequence. If every finitely generated $m$-primary ideal is contained in a $d$-generated ideal then $Tordim(R/m) \leq d$. 
\end{lemma}
\begin{proof}
This is \cite[Lemma 3.8]{AH97}.
\end{proof}
\begin{corollary}
  Let $(R,m)$ be a complete local domain of dimension $2$ such that every finitely generated $m$-primary ideal of $R^{+}$ is contained in a $2$-generated ideal then $Tordim(R^{+}/m_{R^{+}}) \leq 2$. 
\end{corollary}
\begin{proof}
  It is not known whether the conditions of this Corollary are satisfied and by an easy induction argument it is enough to show this for ideals generated by $3$ elements. $R^{+}$ is a limit of two dimensional normal and hence Cohen-Macaulay rings. Thus every pair of elements generating a height $2$ ideal in $R^{+}$ is an $R^{+}$ sequence. Now we may apply the above lemma.
\end{proof}

\begin{lemma}
 If $z$ satisfies a polynomial of the form $X^{n} - f(u, v)$ then $(u,v, z)$ is contained in a two-generated ideal.
\end{lemma}
\begin{proof}
This is \cite[Lemma 3.10]{AH97}.
\end{proof}

We observe the following:
\begin{proposition}
 Let $R$ be a complete local domain of characteristic zero. Assume that the answer to Question \ref{AHq} is yes. Then the answer to Question \ref{BIMq} is yes.  
\end{proposition}
\begin{proof}
$Tor_{i}^{R}(R^{+}, k) = 0$ implies $Tor_{i}^{R}(S, k) = 0$ where $S$ is any normal module finite extension of $R$ (since $S$ is a splinter if it is normal in characteristic zero). This implies $S$ has finite flat dimension at most $i$. Since $R^{+}$ is a limit of all such $S$ we have $R^{+}$ has finite flat dimension (as an $R$ module). Hence we have $Tor_{i}^{R}(R^{+}, k) = 0$ implies $Tor_{j}^{R}(R^{+}, k) = 0$ for any $j\geq i$. The corresponding statement in non-zero characteristics follows from homological properties of perfect(oid) rings and Cohen-Macaulayness of $R^{+}$ when $i \le dim(R)$. This now follows from \cite[Corollary 2.4]{BIM19} with $S = U = R^{+}$.
\end{proof}

\begin{remark}
If one goes through the proof of Theorem \ref{BIMpm} one can verify that for $i>dim(R)$ the proof does \emph{not} use the Cohen-Macaulayness of $R^{+}$ and the theorem holds for any perfectoid ring, not just $R^{+}$. The (almost) Cohen-Macaulayness of $R^{+}$ is used for low $i$'s and in particular for $i=1$. At the time of publication or appearance of the preprint containing \cite[Theorem 4.13]{BIM19} Cohen-Macaulayness of $R^{+}$ was not known in mixed characteristic in dimension $3$. In particular at the time of publication of \cite[Theorem 4.13 3)]{BIM19} was a theorem with vacuous hypothesis: one needs Cohen-Macaulayness of $R^{+}$ to conclude that $Tor_{i}^{R}(R^{+}, k) = 0$ when $dim(R)=3$ and $R$ is a excellent regular local ring of mixed characteristic. 
\end{remark}

It follows from our techniques that $R^{+}$ has flat and projective dimension $dim(R) - 2$ when $R$ is regular. For instance $R^{+}$ is flat when $R$ is regular of dimension $2$. It also follows that the converse is equivalent to the question of Bhatt, Iyengar and Ma. In other characteristics, $R^{+}$ or its $p$-adic completion are flat over $R$ and hence the flat dimension is $0$. However, astonishingly the projective dimension of $R^{+}$ as a $R$ module seems to be unknown in other characteristics.

\begin{question}\label{pdq}
    Let $R$ be an complete local domain of positive characteristic. What is the projective dimension of $R^{+}$ as an $R$ module?
\end{question}

This question has appeared in literature before. André and Fiorot ask whether there is a countably generated $S$-algebra which is $R$-free given $R \rightarrow S$ is a finite extension of complete local domains \cite[Remark 6.3]{AF21}. If the answer to Question \ref{pdq} is zero then it answers this question positively.

\section{The main theorem}

Here is our main theorem, which as promised makes a use of almost mathematics in the sense of the theorem of Roberts, Singh, and Srinivas which states that the image of $H^{2}_{m}(R)_{\geq0}$ in $H^{2}_{m}(R^{+})$ is annihilated by elements of arbitrarily small positive degree (is `almost zero').
\begin{customthm}{A}\label{thm:main}
Let $R$ be a $2$-dimensional $\mathbb{N}$-graded ring finitely generated over a characteristic zero field. If $Tor_{i}^{R}(R^{+},k) = 0$ for some $i\geq 1$ then $R$ is regular.
\end{customthm}

\begin{proof}
We will first use elementary commutative algebra to show $R$ is normal. Let $R^{n}$ be the normalization of $R$. Since normal rings are splinters we have that the map $R^{n} \rightarrow R^{+}$ splits (as $R^{n}$ and hence $R$ module map). Since $Tor$ is functorial $Tor_{i}^{R}(R^{n},k)$ is a direct summand of $Tor_{i}^{R}(R^{+},k) = 0$ and is $0$. As $R$ is excellent $R^{n}$ is module finite over $R$ and hence the projective dimension of $R^{n}$ is finite over $R$. The Auslander-Buschbaum formula then implies that $R^{n}$ must be free and in particular flat over $R$. Normality descends along flat maps which implies that $R = R^{n}$ and hence $R$ is normal.

Now we come to the key innovative idea in this work. Let $S$ be a finite normal extension of $R$. Since normal rings are splinters we have that the map $S \rightarrow R^{+}$ splits (as $S$ and hence $R$ module map). Since $Tor$ is functorial $Tor_{i}^{R}(S,k)$ is a direct summand of $Tor_{i}^{R}(R^{+},k) = 0$ and is $0$. Hence $S$ is flat as an $R$-module. $R^{+GR}$ is normal and hence is flat as it is a direct summand of the flat module $R^{+}$. Let $[\alpha] \in [H^2_{m}(R)]_{\geq 0}$ (degree $\geq 0$ part of the top local cohomology). Since $R \rightarrow R^{+GR}$ is flat, local cohomology and consequently it's annihilators base change, so we have $ann_{R}([\alpha]) \otimes R^{+GR} = ann_{R^{+GR}}([\alpha])$. However \cite[Corollary 3.5]{RSS07} states that there are elements of arbitrarily small positive degree in $ann_{R^{+GR}}([\alpha])$, and $ann_{R}([\alpha]) \otimes R^{+GR} = ann_{R^{+GR}}([\alpha])$ is clearly finitely generated. Hence all elements in have a degree greater than the minimum of all of it's generators. Hence $ann_{R^{+GR}}([\alpha])$ contains $1$ and consequently $[\alpha] = 0$. It is well known that this implies $R$ has rational singularities. Commutative algebraists can observe that this precisely says that the $a$-invariant of $R$ is negative and since $R$ is a normal ring of dimension $2$ the statement follows from \cite[Theorem 2.2]{Wat83}. Algebraic geometers can check that criterions for cones over smooth projective varieties \cite{Kol13} work in this case. 

Next, we show that $R$ is Gorenstein. Since $R$ is a rational surface singularity we know that it is $\mathbb{Q}$-Gorenstein \cite[Theorem 17.4]{Lip78} and that it's cyclic cover $S$ is normal and Gorenstein (standard theory of cyclic covers). Note that $S$ need not have rational singularities (atleast we do not claim so). Since normal rings are splinters we have that the map $S \rightarrow R^{+}$ splits (as $S$ and hence $R$ module map). Since $Tor$ is functorial $Tor_{i}^{R}(S,k)$ is a direct summand of $Tor_{i}^{R}(R^{+},k) = 0$ and is $0$. Hence $R \rightarrow S$ is flat. The property of a ring being Gorenstein descends along a flat map \cite[Theorem 23.4]{Mat89} we have that $R$ is Gorenstein.

The desired conclusion follows from classical work on rational double points. \cite{Pri67}, \cite{Lip78} say that $R$ has a module finite extension $S$ such that $S$ is regular. Arguments in above paragraphs imply that $S$ is flat over $R$ which implies $R$ is regular since regularity descends under flat maps.
\end{proof}

\begin{remark}
Theorem \ref{thm:main} implies $R$ is in fact a polynomial ring (see \cite[Appendix III, 3, Theorem 1]{Ser00}). 
\end{remark}

We discuss the information our proof gives us and the sharpness of our ideas. Let us examine it in more detail.
\par

The following propositions are easily seen to pop out of our proof above:

\begin{proposition} \label{p1}
Let $R$ be an excellent characteristic zero noetherian local domain such that $Tor_{i}^{R}(R^{+}, k)= 0$ for some $i \geq 1$. Then $R$ is normal.
\end{proposition}
\begin{proof}
$R$ has a module finite normal extension $S$ and since $S$ has characteristic zero it is a direct summand of $R^{+}$. Hence we have that $Tor_{i}^{R}(S,k) = 0$. Hence $S$ has finite projective dimension. Let $\mathfrak{p}$ be a height two prime ideal of $R$. We have $S_{\mathfrak{p}}$ is normal (normality localizes) has finite projective dimension over $R_{\mathfrak{p}}$.  Auslander-Buschbaum formula implies projective dimension of $S_{\mathfrak{p}}$ over $R_{\mathfrak{p}}$ must be zero and hence $S_{\mathfrak{p}}$ is faithfully flat over $R_{\mathfrak{p}}$ and hence $R_{\mathfrak{p}}$ is normal. This implies $R$ is normal as it is enough to check it on localization at height $2$ primes.
\end{proof}

\begin{proposition} \label{p2}
Let $R$ be a noetherian characteristic zero complete local domain such that $Tor_{i}^{R}(R^{+}, k)= 0$ for some $i \geq 1$. Suppose $R$ is a module finite direct summand of a regular local ring $S$. Then $R$ is regular.
\end{proposition}
\begin{proof}
Since $S$ is normal and has characteristic zero it is a direct summand of $R^{+}$. Hence $Tor_{i}^{R}(S, k) = 0$ for some $i$ and we have that $S$ has finite projective dimension over $R$. The Auslander Buschbaum formula implies that the projective dimension must be $0$ and consequently $S$ is flat as a $R$ module. Faithfully flat descent of regularity implies $R$ must be regular itself.
\end{proof}

\begin{theorem}
Let $(R,m,k)$ be a complete local domain with a characteristic zero algebraically closed residue field (this implies $R$ has characteristic zero). If $Tor_{i}^{R}(R^{+}, k) = 0$ then $R$ is Gorenstein.
\end{theorem}
\begin{proof}
We have that $R$ has a module finite extension $S$ which is local, normal, and Gorenstein by \cite[Proposition 2.4]{ST21}. Since normal rings are splinters we have that the map $S \rightarrow R^{+}$ splits (as $S$ and hence $R$ module map). Since $Tor$ is functorial $Tor_{i}^{R}(S,k)$ is a direct summand of $Tor_{i}^{R}(R^{+},k) = 0$ and is $0$. Hence $R \rightarrow S$ is flat. The property of a ring being Gorenstein descends along a flat local map \cite[Theorem 23.4]{Mat89} we have that $R$ is Gorenstein.
\end{proof}

On the way to prove that the ring is regular, three out of the four steps in our proof use the fact that $R^{+}$ is `ind-$X$' where $X$ is normal, Gorenstein, and given that it is a rational double point, regular. The starting point of our result was when inspired by the preprint of Shimomoto and Tavanfar after which the author asked is $R^{+}$ `ind pseudo-rational'? The answer is no, and for a good reason. Boutot's theorem implies that a normal non-pseudorational ring has \emph{no} module finite extensions which are pseudo-rational since normal rings are splinters in characteristic zero. Hence new ideas are required and the theorem of Roberts, Singh, and Srinivas comes into play. In a sense, one has to use all of $R^{+}$ to show that a ring $R$ satisfying the hypothesis of the question of Bhatt, Iyengar, and Ma has rational singularities. This suggests deep connections between the absolute integral closure and rational singularities in all characteristics. In particular this indicates that dagger closure might give yet another algebraic characterization of rational singularities in characteristic zero. In dimension $2$ our proof uses the main theorems about rational surface singularities. It seems from our proof and that certain higher dimensional analogues of the theorem of Roberts, Singh, and Srinivas (see Section $5$ below) can be used to show $R$ has rational singularities in higher dimensions if it satisfies the hypothesis of Question \ref{BIMq}. To show $R$ is regular we expect theorems such as \cite[Prop. 3.5]{MS20} and proof of \cite[Theorem 3.6]{MS20} to be useful.

\begin{remark}
In a similar spirit additional obstructions to having a regular ring as a module finite extension are studied in Devlin Mallory's work \cite{Mal21}, \cite{Mal22}. For example it follows from an elementary analysis of divisor class groups or picard groups that $k[x_{1}, x_{2}, x_{3}, x_{4}]/(x_{1}^{3} + x_{2}^{3} + x_{3}^{3} + x_{4}^{3})$ is not a module finite direct summand of a polynomial ring whereas one needs to use the more technical machinery of differential operators to show that it is not a summand of an arbitrary polynomial ring.
\end{remark}

While the answer to Question \ref{BIMq} in higher dimensions is currently out of reach, to answer it for all excellent noetherian domains it seems practical to make the following definition.
\begin{definition}
A ring $R$ is said to be of \textbf{BIM-type} if every module-finite extension of $R$ has finite projective dimension over it. A ring $R$ is said to be of \textbf{normal BIM-type} if every module-finite normal extension of $R$ has finite projective dimension over it.
\end{definition}
We expect noetherian domains to be of BIM-type if and only if they are regular. Question \ref{BIMq} precisely asks whether $R$ is regular if $R$ is an excellent characteristic zero noetherian local domain of normal BIM-type. It follows from Kunz's theorem if $R$ is $F$-finite of positive characteristic then and by \cite[Theorem 4.13]{BIM19} if $R$ is excellent local of mixed characteristic that $R$ is of BIM-type if and only if it is regular. It follows from the arguments in this section that excellent noetherian rings of BIM-type are normal.

We expect the following question to be of importance while trying to extend positive results for Question \ref{BIMq} from nice classes of rings such as complete local or graded domains to all excellent noetherian local domains.
\begin{question}
What permanence properties does the property a ring being of BIM-type have? For excellent noetherian local rings, does this property descend along faithfully flat maps? Does it ascend along completions?
\end{question}

\section{Higher dimensions}

We believe our proof of the main theorem creates sufficient evidence to suggest that higher dimensional analogues of \cite[Theorem 3.4]{RSS07} and the theory of rational surface singularities, for example \cite[Section 3]{MS20} will make progress towards Question \ref{BIMc}. In this section we attempt to make this precise as some of our techniques are clearly restricted to dimension $2$. In particular we list several questions asking for statements for future study. These naturally lead to questions open even in positive characteristic. 

The proof of our main theorem uses the assumption that $R$ has dimension $2$ in an essential way. We crucially use that the assumptions of Question \ref{BIMq} imply that the map $R\rightarrow R^{+}$ is flat so that we can base change local cohomology to $R^{+}$. This fails in higher dimensions. It was shown in Section 2 that there are \emph{no} rings $R$ (under mild assumptions) such that $R\rightarrow R^{+}$ is flat. This motivates us to ask for higher dimensional analogues of theorems of Shimomoto and Tavanfar and Roberts, Singh, Srinivas and  with \emph{depth constraints}. The depth constraints will force maps to be flat (because of the Auslander-Buschbaum formula). 

\begin{question}\label{STdepth}
Let $R$ be a $\mathbb{N}$-graded domain, finitely generated over a field $R_0$ of characteristic zero. Is there a finite extension of $R$, $S$, such that \textbf{$depth(S) = depth(R)$} and $S$ is normal and quasi-Gorenstein?
\end{question}

\begin{question}\label{RSSdepth}
Let $R$ be an $\mathbb{N}$-graded domain, finitely generated over a field $R_0$ of characteristic zero. For $i<\dim R$, and $\epsilon \in \mathbb{R}_{\ge 0}$ is there a finite extension of $R$, $S$, such that \textbf{$depth(S) = depth(R)$} and
\[
H^i_{\frakm}(R)\to H^i_{\frakm}(S) 
\]
killed by elements of $R^{+GR}$ of positive degree $d \le \epsilon$?

and here is the `Segre product' version of the above question:

Let $R$ be an $\mathbb{N}$-graded domain of dimension $d$, finitely generated over a field $R_0$ of characteristic zero. For $\epsilon \in \mathbb{R}_{\ge 0}$ is there a finite normal extension of $R$, $S$, such that such that $depth(S) = depth(R)$ and the image of
\[
\left[H^d_{\frakm}(R)\right]_{\ge0}\to H^d_{\frakm}(S)
\]
killed by elements of $R^{+GR}$ of arbitrarily small positive degree?

\end{question}

\par

It is possible that for a ring satisfying the assumptions of Question \ref{BIMq} a positive answer to the original question of Roberts, Singh, and Srinivas (which is precisely Question \ref{RSSdepth} without the depth constraint \textbf{$depth(S) = depth(R)$} and $R^{+}$ in place of $S$) gives a positive answer to Question \ref{RSSdepth}, we do not know. However this already yields the following natural question in positive characteristic which we believe to be interesting. If true it would be a refinement of a result of Huneke and Lyubeznik which has had many applications to positive characteristic algebraic geometry and singularities.

\begin{question}\label{HLdepth}
Let $(R, m)$ be an excellent local domain of positive characteristic. A famous result of Huneke and Lyubeznik states that there exists a module finite extension $S$ of $R$ has such that the map from $H^{i}_{m}(R) \rightarrow H^{i}_{m}(S)$ is zero for all $i<dim(R)$. Does there exist such an $S$ such that $depth(S) = depth(R)$ ?
\end{question}

The condition $depth(S) = depth(R)$ is sharp. It is likely not true that there will exist an $S$ such that $depth(S) \ge depth(R) + 1$. For example Bhatt shows, using Witt vector cohomology, that cones over abelian surfaces (which are normal rings of dimension $3$) do not have any module finite Cohen-Macaulay ring extensions. It is not hard to see that the answer to the above question is yes if $dim(R) = 2$ (by simply taking normalization), if the Frobenius kills all local cohomologies (i.e. if $R$ is `$F$-nilpotent'), or if $R$ is the cone over an abelian variety (by using the multiplication by $p$ map). We leave these as an excercise to the reader. 

\begin{remark}
Even positive answers to above questions are not enough to conclude properties of $R$ in higher dimensions. For example, to apply our main trick in dimension $2$ involving the interplay between almost mathematics and flat maps, we need the extensions in Question \ref{RSSdepth} to form a direct limit system or to dominate each other. That is it would be necessary, or atleast useful, to have a ring $P$ satisfying the conditions in Question \ref{RSSdepth} such that $S$ and $T$ map to $P$ for any $S$ and $T$ satisfying the conditions in Question \ref{RSSdepth}. We do not make this precise here for the sake of brevity and clarity.
\end{remark}

\section{A question of André and Fiorot}
The purpose of this section is to demonstrate that a question of André and Fiorot is related to Question \ref{BIMq}.  We begin by stating the question.

\begin{citedque}[{\citeleft\citen{AF21}\citemid Question 10.1\citeright }]\label{AFq} 
Which affine Noetherian schemes have the property that every finite covering is a covering for the fpqc topology? (Equivalently: which Noetherian
rings $R$ have the property that for every finite extension $S$, there is an $S$-algebra.
faithfully flat over $R$?)
\end{citedque}

We begin with the following observation:

\begin{proposition}
Let $R$ be an excellent noetherian domain which is a ``fpqc analogue of splinter". Then $R$ is normal. 
\end{proposition}
\begin{proof}
Let $S$ be the normalization of $R$. By assumption we have a faithfully flat $R$ algebra $T$ containing $S$. This is easily seen to imply that $R \rightarrow T$ and hence $R \rightarrow S$ is pure. Since $S$ is normal we have $R$ is normal.
\end{proof}

\begin{proposition}
(see \cite[Theorem 10.4 2) and 1)]{AF21}) Let $R$ be an excellent noetherian domain which is a `fpqc analogue of splinter'. If $R$ has a module finite extension which is a regular local  ring $S$ then it is regular itself. 
\end{proposition}
\begin{proof}
By assumption we have a faithfully flat $R$ algebra $T$ containing $S$. $R \rightarrow S$ splits since $R \rightarrow T$ is pure. This implies $R$ is Cohen-Macaulay (and in fact has rational singularities by Boutot's theorem). Since $T$ is faithfully flat over $R$ it follows that $T$ is a big Cohen-Macaulay $R$ algebra. Since $S$ is a module finite extension of $R$. $T$ is also a big Cohen-Macaulay $S$ algebra. Since $S$ is regular $T$ is faithfully flat over $S$. Hence $R \rightarrow S$ is faithfully flat and by descent it follows that $R$ is regular. 
\end{proof}

Note the parallels to propositions \ref{p1}, \ref{p2} from Section 4. 

\begin{remark}
In fact, the above proposition is true even without the module finite assumption and follows from a theorem of Bhatt, Iyengar, and Ma. See \cite[Theorem 10.4]{AF21}. We are investigating whether \cite[Theorem 10.4]{AF21} can be used to answer Question \ref{BIMq} under these assumptions (that $R$ has a ring extension $S$ which is regular).
\end{remark}

\begin{theorem}
Assume the following:
\begin{itemize}
    \item $R$ is a complete local domain which is completion of a $\mathbb{N}$-graded ring of dimension $2$ (at the homogenous maximal ideal).
    \item every faithfully flat $R$ algebra $T$ maps to a permissible Cohen-Macaulay algebra in the sense of \cite{HH95}. 
\end{itemize}
Assume $R$ is a fpqc analogue of a splinter. Then $R$ has rational singularities. Additionally if $R$ is Gorenstein then it is regular.   
\end{theorem}
\begin{proof}
We may assume $R$ itself is the $\mathbb{N}$-graded ring while talking about local cohomology since it is invariant under completion. Let $\alpha$ be in $[H^2_{m}(R)]_{\geq 0}$ and $\frac{y}{x_{1}...x_{m}}$ a fraction representing $\alpha$. Let $v$ be the valuation arising from grading. Translating the cohomological statement from \cite[Theorem 3.6]{RSS07} to containment of  implies that $R$ has a module finite extension $S$ inside $R^{+GR}$ such that there is an element $z \in R^{+GR}$ with $v(z) < min_{i} v(y_{i})$ where $y_{i}$ generate the colon ideal $(y:_{R} (x_{1},...,x_{m}))$ and $z \in (y:_{R^{+GR}} (x_{1},...,x_{m}))$. Note that R is excellent. Since $R$ is a fpqc analgoue of a splinter we have a flat $R$ algebra $T$ which contains $S$ and which maps to a permissible big Cohen-Macaulay algebra $U$. Since colon ideals base change under flat maps we have the same containment in $U$. By works of Hochster-Huneke \cite[Theorem 5.6 (a)]{Ho94} and \cite[Theorem 5.12]{HH95}, $z$ is then contained in the big equational tight closure of $(y_{1}, . . . , y_{n})$. Hence, it is a fortiori contained in the integral closure of $(y_{1}, . . . , y_{n})$. Now the result follows by the characterisation of integral closure in terms of valuations (\cite[Theorem 10.2.4 a)]{HS18}). If one considers the valuation induced by the grading we get that $v(z) \geq min_{i} v(y_{i})$ however using the theorem of Roberts, Singh, and Srinivas we choose $z$ to have valuation lower than $v(y_{i})$ for all $i$. Hence the colon ideal must be the entire ring and the cohomology class must be $0$. Hence by \cite[Theorem 2.2]{Wat83} we have that the ring has rational singularities. If in addition $R$ is Gorenstein then it is a rational double point and the result follows from propositions above (rational double points have module finite extensions which are regular). 
\end{proof}

Techniques inspiring the theorem above can be found in \cite[Lemma 2.2]{BS12}. This bolsters our philosophy that our main result and answers to Question \ref{BIMq} can be considered as analogues of Kunz's theorem in charactersitic zero as the fact that (F-finite) fpqc analogues of splinters are regular rings in positive characteristic follows from Kunz's theorem. 

\begin{remark}
We do not know whether one can conclude that $R$ is Gorenstein with the assumptions (1) and (2) above. Conceptually speaking this is because the author is not aware of a way to detect Gorensteinness using closure operations similar to rationality. 
\end{remark}

\begin{remark}
It is possible that the above proof will go through if we replace permissible Cohen-Macaulay algebras with Schouten's big Cohen-Macaulay algebras \cite{Sch03} (big Cohen-Macaulay algebras which arise as ultraproducts of positive characteristic big Cohen-Macaulay algebras) in place of permissible big Cohen-Macaulay algebras. 
\end{remark}

\begin{remark}
The assumption `maps to a permissible big Cohen-Macaulay algebra' is perhaps not that surprising. Due to the absence of Frobenius in characteristic zero and mixed characteristic authors have had to impose conditions on big Cohen-Macaulay algebras (to make them more controllable) in the past. See \cite[footnote on p. 37]{CLM+22} and first arxiv versions of \cite{MS21}. The recent work T. Yamaguchi \cite{Ya22} tries to relate arbitrary big Cohen-Macaulay algebras in characteristic zero with Schouten's big Cohen-Macaulay algebras \cite{Sch03}.
\end{remark}

\begin{remark}
Fpqc analogues of splinters behave well with ultraproducts and in particular ascend under completion. However it is not known whether they are compatible with Hochster's $\Gamma$ construction and hence we do not know whether one can omit $F$-finiteness from the positive characteristic statement that $F$-finite fpqc analogues of splinters are regular rings. We are grateful to S. Lyu for this remark.   
\end{remark}

\section{Acknowledgements}
The author thanks Mohsen Asgharzadeh, Bhargav Bhatt, Neil Epstein, Arnab Kundu, Shiji Lyu, Linquan Ma, Devlin Mallory, Alapan Mukhopadhyay, Swaraj Pande, Vaibhav Pandey, Sambit Senapati, Kazuma Shimomoto, Austyn Simpson, Sridhar Venkatesh and his advisor Kevin Tucker for conversations, encouragement, friendship, mentorship, useful suggestions, or a reading of the manuscript. The author especially thanks Wenliang Zhang for bringing his attention to Theorem \ref{BIMpm} in Fall 2019 when the mixed characteristic statement was not known (in dimensions $4$ or more).

\end{document}